\documentclass[10pt]{article}

\pdfoutput=1

\usepackage{amssymb}
\usepackage{amsfonts}
\usepackage{amsthm}
\usepackage{amsmath}
\usepackage{eufrak}
\usepackage{mathtools}

\newtheorem{Lemma}{Lemma}

\newtheorem{Definition}{Definition}
\newtheorem{Theorem}{Theorem}
\newtheorem{Corollary}{Corollary}
\usepackage{hieroglf}

\usepackage{slashed}
\usepackage{fullpage}
\usepackage{tikz}
\usetikzlibrary{matrix}
\usepackage{scalefnt}
\newcommand\ScaleExists[1]{\vcenter{\hbox{\scalefont{#1}$\exists$}}}

\DeclareMathOperator*\bigexists{%
  \vphantom\sum
  \mathchoice{\ScaleExists{2}}{\ScaleExists{1.4}}{\ScaleExists{1}}{\ScaleExists{0.75}}}

\RequirePackage{mdwtab,latexsym,amsmath,amsfonts,ifthen,wasysym,cancel}

\usepackage{graphicx}
\graphicspath{{C:/Users/Sam/Desktop/PaperFigures/}}
\DeclareGraphicsExtensions{.pdf}

\begin{document}
\title{On Generalizing a Temporal Formalism for Game Theory to the Asymptotic Combinatorics of S$5$ Modal Frames}
\author{Samuel Reid}
\date{\today}
\maketitle

\begin{abstract}
A temporal-theoretic formalism for understanding game theory is described where a strict ordering relation on a set of time points $T$ defines a game on $T$. Using this formalism, a proof of Zermelo's Theorem, which states that every finite 2-player zero-sum game is determined, is given and an exhaustive analysis of the game of Nim is presented. Furthermore, a combinatorial analysis of games on a set of arbitrary time points is given; in particular, it is proved that the number of distinct games on a set $T$ with cardinality $n$ is the number of partial orders on a set of $n$ elements. By generalizing this theorem from temporal modal frames to S$5$ modal frames, it is proved that the number of isomorphism classes of S$5$ modal frames $\mathcal{F} = \langle W, R \rangle$ with $|W|=n$ is equal to the partition function $p(n)$. As a corollary of the fact that the partition function is asymptotic to the Hardy-Ramanujan number
$$\frac{1}{4\sqrt{3}n}e^{\pi \sqrt{2n/3}}$$
the number of isomorphism classes of S$5$ modal frames $\mathcal{F} = \langle W, R \rangle$ with $|W|=n$ is asymptotically the Hardy-Ramanujan number. Lastly, we use these results to prove that an arbitrary modal frame is an S$5$ modal frame with probability zero.
\end{abstract}

\section{Temporal Syntax and Semantics}
Temporal Logic provides a comprehensive framework for understanding time through the use of model theory and modalities. In particular, we can understand linear, branching, discrete, dense, or Dedekind Complete time flows by simple conditions on a linear ordering in a model. For applications to game theory, we will be interested in branching time flows.

\begin{Definition}
Let $\mathcal{F} = \langle T, < \rangle$ be a frame where $T$ is a set of time points determining a time flow and $<$ is a strict ordering relation, that is, $<$ is antireflexive, antisymmetric, and transitive. 
\end{Definition}

We now define the ordering relation which is used in branching temporal logic to define a ``tree of time''.

\begin{Definition}
A branching frame $\mathcal{F} = \langle T, < \rangle$ is a frame where $\{s \; | \; s < t \}$ is strictly partially ordered by $<$ for all $t \in T$.
\end{Definition}

For discussing one particular branch in the tree of time we consider histories in $T$ in order to be able to determine if modal formulas are satisfied when referring to only a single branch of the tree of time.

\begin{Definition}
Given a branching frame $\mathcal{F} = \langle T, < \rangle$, a history $h \in H$ in $T$ is a maximal linearly ordered subset of $T$, where $H$ is the set of all histories in $T$.
\end{Definition}

A branching model of time results by assigning a valuation to the time points of a branching frame.

\begin{Definition}
A branching model of time $\mathcal{M} = \langle T, <, V \rangle$ is a branching frame with a valuation $V: T \times \text{Var} \rightarrow \{0,1\}$.
\end{Definition}




We can now define the traditional alethic modalities of necessity and possibility as follows.

\begin{Definition}
Let $\varphi$ be an arbitrary formula and define the two alethic modalities, necessity ($\Box$) and possibility ($\Diamond$), as follows:
\begin{itemize}
\item $\mathcal{M}_{h,t} \vDash \Box \varphi$ if and only if $\forall t' > t, \mathcal{M}_{h,t'} \vDash \varphi$ when $\neg \exists t''$ so that $t < t'' < t'$ with $t,t',t'' \in T$.
\item $\mathcal{M}_{h,t} \vDash \Diamond \varphi$ if and only if $\exists t' > t, \mathcal{M}_{h,t'} \vDash \varphi$ such that $\neg \exists t''$ so that $t < t'' < t$ with $t,t',t'' \in T$.
\end{itemize}
\end{Definition}

We then have that necessity gives us a condition on all possible time flows emanating into the future from one time point and that possibility gives us a condition on at least one time flow emanating into the future from one time point, as is illustrated in the following figure.

\begin{figure}[h!]
\begin{center}
\includegraphics[scale=0.39]{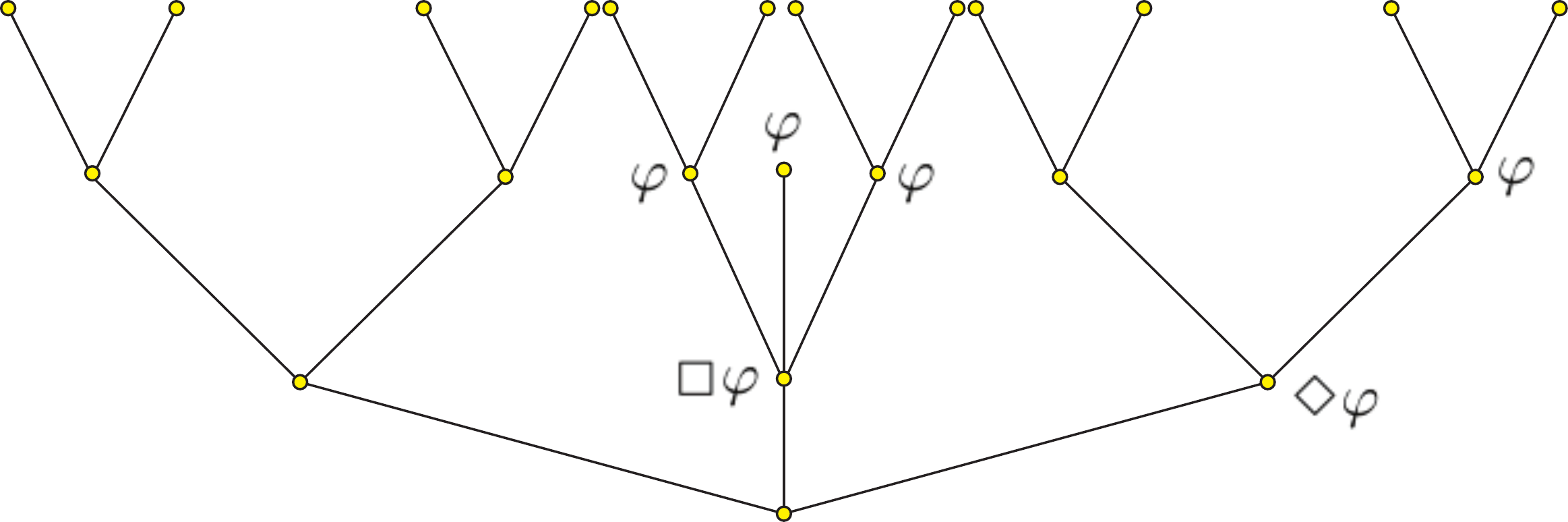}
\caption{An illustration of a branching model $\mathcal{M}$ with certain modal formulas satisfied.}
\end{center}
\end{figure}

We now consider articulated histories, which split histories into the past time points and possible future time points.

\begin{Definition}
An articulated history of a time point $t \in T$ is a pair $(h_{p}(t),h_{f}(t))$, where $h_{p}(t) = \{ t' \; | \; t' < t\}$ is the set of all past time points of $t$ in all histories containing $t$ and $h_{f}(t) = \{ t' \; | t' > t\}$ is a set of future time points of $t$ which determines a unique history $h = h_{p}(t) \cup \{t\} \cup h_{f}(t)$.
\end{Definition}

Using the idea of articulated histories, the set of time points $T$ can be partitioned into time points which are considered to coincide at the same instant.

\begin{Definition}
A set of time points $\{t_{1},...,t_{n}\}$ belongs to an instant $I \subseteq T$ if $t_{i} \slashed{<} t_{j}$ and $|h_{p}(t_{i})| = |h_{p}(t_{j})|, \forall i,j$.
\end{Definition}

\section{A Temporal-Theoretic Formalism for Game Theory}
The temporal syntax and semantics constructed in the previous section can now be applied to game theory by considering time flows, or histories, as ways in which a game is played, time points as turns, and a partition of the instants of the time points into turns associated with a player in the game. We now call the branching model of time $\mathcal{M}$ a game $G$ and let $\psi$ be a formula which denotes that the game is a tie.

\begin{Definition}
In an $n$-player game $G$, the $i^{\text{th}}$ instant $I_{i}$ is player $k$'s turn if $k \equiv i \mod n$. Furthermore, player $k$ wins the game at a turn $t$ in a history $h$ if $|h_{f}(t)| = 0$ and $t \in I_{i}$ with $k \equiv i - 1 \mod n$ and $G_{h,t} \vDash \neg \psi$.
\end{Definition}

Simply put, player $k$ wins if the game ends in a non-tie after player $k$'s turn, as $|h_{f}(t)|=0$ means that there are no future turns to be played. The combinatorial analysis of $G$ can be simplified by considering formulas which will be satisfied at a turn $t$ if player $k$ wins or loses at a turn $t$. That is, define the formula $\omega_{k}$ to mean that player $k$ wins and define the formula $\chi_{k}$ to mean that player $k$ loses as follows
$$\omega_{k} := \neg \psi \wedge t \in I_{i} \wedge |h_{f}(t)| = 0 \wedge k \equiv i-1 \mod n$$
$$\chi_{k} := \neg \psi \wedge (t \notin I_{i} \vee |h_{f}(t)| \neq 0 \vee k \slashed{\equiv} i-1 \mod n)$$

It is clear that $\psi \Rightarrow (\neg \omega_{k} \wedge \neg \chi_{k})$, which is that a tie implies player $k$ did not win and that player $k$ did not lose. Observe that a non-tie does not imply anything about player $k$ winning or losing as we have that there exists a 2-player game $G$ such that $\exists t \in T$ for which
$$G_{h,t} \vDash \Diamond(\neg \psi \wedge \neg(\omega_{k} \vee \chi_{k}))$$

Namely, this holds for $k=1$ in the following 2-player game.

\begin{figure}[h!]
\begin{center}
\includegraphics[scale=0.39]{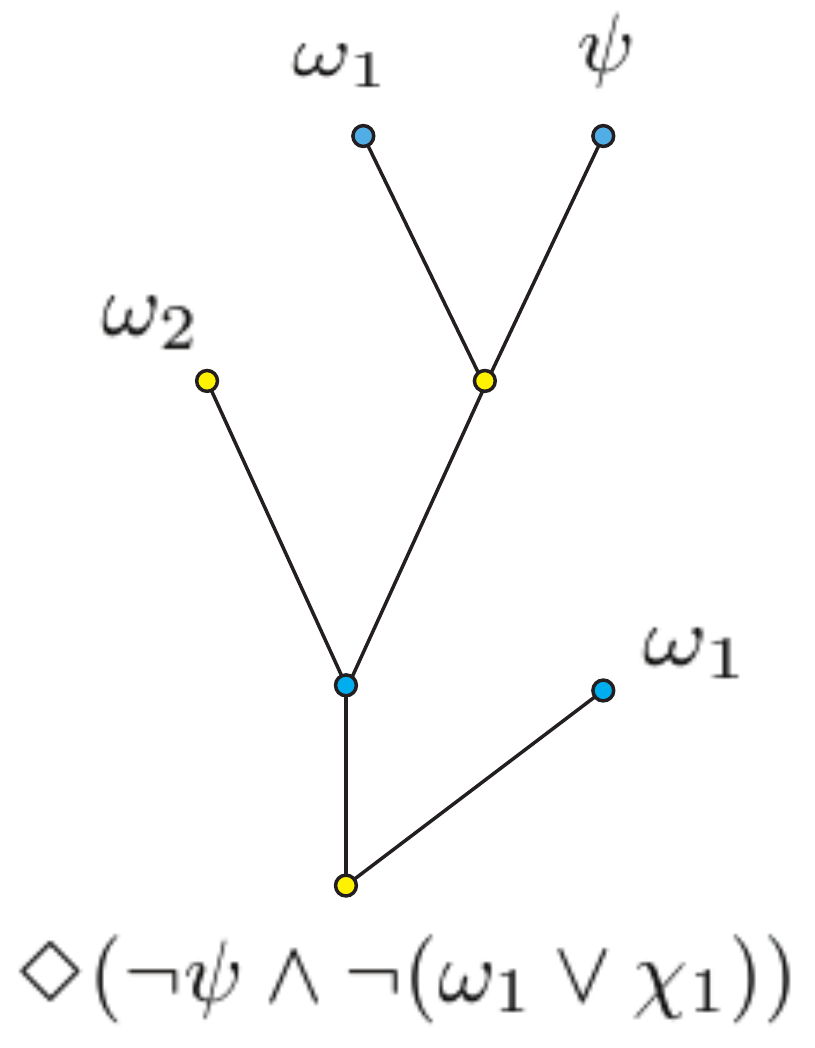}
\caption{A 2-player game where $G_{h,t_{0}} \vDash \Diamond(\neg \psi \wedge \neg(\omega_{1} \vee \chi_{1}))$, where the yellow instants are player 1's turns and the blue instants are player 2's turns.}
\end{center}
\end{figure}

The notions of winning strategy and drawing strategy can now be defined by noticing the cardinalities of the future time points of an articulated history. Intuitively, a winning strategy for player $k$ is a response to the moves of the previous player so that player $k$ can win no matter what moves the other players make; similarly, a drawing strategy is a strategy which does not guarantee a win for a particular player, but guarantees that that player does not lose.

\begin{Definition}
In an $n$-player game $G$, a winning strategy exists for player $k$ if for every sequence of moves $(t_{k-1},t_{2(k-1)},...,t_{l(k-1)})$ by player $k-1$, there exists a sequence of moves $(t_{k},t_{2k},...,t_{lk})$ by player $k$ such that
$$G_{h,t_{i}} \vDash \Diamond^{|h_{f}(t_{i})|} \omega_{k}$$
$\forall i=jk, 1 \leq j \leq l$, where $t_{i} \in I_{i}$ and $k \equiv i \mod n$.
\end{Definition}

The temporal formalism developed for discussing a branching model of time can be applied to give a proof of Zermelo's Theorem, a classic result in game theory originally proved in a German paper from 1913 by Zermelo. An English presentation of the proof in modern notation is given in \cite{Schwalbe2001123}. A finite game is a game that must necessarily terminate in a finite number of moves, that is, $|h_{f}(t_{0})| < k \in \mathbb{N}$ for all future histories of $t_{0}$, where $t_{0}$ is the initial turn defined as satisfying the property that $t_{0} < t$, $\forall t \in T \backslash \{t_{0}\}$. Zero-sum games are games where each player can either win or lose, that is, $G_{h,t} \vDash \neg \psi$, $\forall t \in T$.

\begin{Theorem}
(Zermelo) Every finite zero-sum 2-player game is determined (one of the two players has a winning strategy). 
\end{Theorem}
\begin{proof}
We prove Zermelo's Theorem by induction on the length of the game $L(G)$, defined to be the maximum cardinality of the future time points in an articulated history, that is $L(G) = \max\{|h_{f}(t_{0})| \; | \; h \in H\}$. If $L(G) = 1$ then we have that $G_{h,t_{0}} \vDash \Box \omega_{1}$ because $G$ is zero-sum, and so $G$ is determined. So, assume that every finite zero-sum 2-player game $G$ is determined for $L(G) = k$, where $k \in \mathbb{N}$. Then let $G$ be an arbitrary game with $L(G) = k+1$ and consider the set of all time points on player 2's first turn, denoted by $I_{2}$. Let $t' \in I_{2}$ be arbitrary and notice that $t'$ is the root of the subgame $G' \subset G$ with $L(G') = k$. By the induction hypothesis we then have that $G'$ is determined and so we have that either $G'_{h,t'} \vDash \Diamond^{k} \omega_{1}$ or $G'_{h,t'} \vDash \Diamond^{k}\omega_{2}$. If $G'_{h,t'} \vDash \Diamond^{k} \omega_{1}$ then $G_{h,t_{0}} \vDash \Diamond^{k+1} \omega_{1}$ and if $G'_{h,t'} \vDash \Diamond^{k}\omega_{2}$ then $G_{h,{t_{0}}} \vDash \Diamond^{k+1} \omega_{2}$ and so we have that $G$ is determined. Therefore, by induction we have that every finite zero-sum 2-player is determined.
\end{proof}

\section{Example: The Game of Nim}
In order to illustrate the temporal formalism developed in the preceding section, we give an exhaustive temporal analysis of the game of Nim \cite{modelsandgames}. In the game of Nim, two players called Alice and Bob choose either one or two tokens on their turn from the same pile of six tokens. The winner of the game is the player who removes the last token. Nim is a zero-sum game because a win for Alice is a loss for Bob and a win for Bob is a loss for Alice, finite because there are at most 6 turns, and 2-player by definition. Therefore, Nim is determined by Zermelo's Theorem and so we know that there is a winning strategy for either Alice or Bob. If Alice goes first then Bob has a winning strategy which is as follows: If Alice takes two tokens then Bob takes one token and if Alice takes one token then Bob takes two tokens. Let $Alice(1)$, resp. $Bob(2)$, denote Alice (Bob) taking one token from the pile and let $Alice(2)$, resp. $Bob(2)$, denote Alice (Bob) taking two tokens from the pile. Then we have that the winning strategy for Bob can be formalized as follows. Let $(Alice(x_{1}),Alice(x_{2}),Alice(x_{3}))$ be an arbitrary sequence of moves by Alice, where $x_{1},x_{2},x_{3} \in \{0,1,2\}$. If $x_{i} = 0$ then it means that the game has already ended. Define a function as follows,
$$\chi(x_{i}) =
\begin{cases}
1 \;\;\; \text{if } x_{i} = 2 \\
2 \;\;\; \text{if } x_{i} = 1 \\
0 \;\;\; \text{if } x_{i} = 0
\end{cases}$$
Then the strategy $(Bob(\chi(x_{1})),Bob(\chi(x_{2})),Bob(\chi(x_{3})))$ ensures that $\text{Nim}_{h,t_{i}} \vDash \Diamond^{|h_{f}(t_{i})|} \omega_{\text{Bob}}$, where $t_{i} \in I_{i}$. This illustrates how the satisfaction of a series of modal formulas, namely $\Diamond^{|h_{f}(t_{i})|} \omega_{\text{Bob}}$ at each $t_{i} \in I_{i}$ can capture the notion of a winning strategy in game theory.

\begin{figure}[h!]
\begin{center}
\includegraphics[scale=0.43]{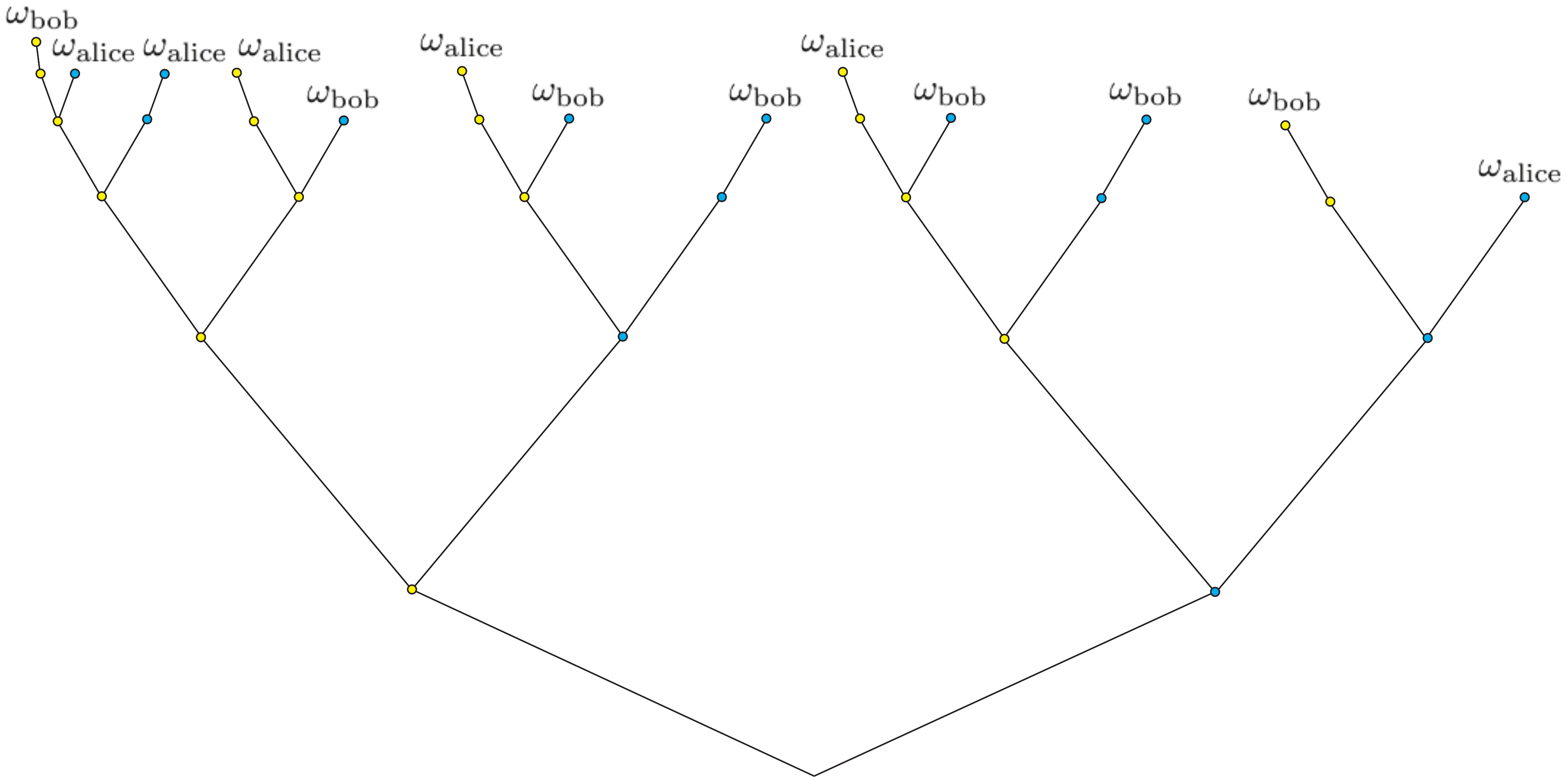}
\caption{All possible histories in the game of Nim where Alice plays first. Moving to a blue node denotes taking two tokens from the pile and moving to a yellow node denotes taking one token from the pile.}
\end{center}
\end{figure}

\section{Combinatorial Analysis of Finite Games and S$5$ Modal Frames}
It is interesting to observe the asymptotic combinatorial properties of any structure, and in particular, we can comment on games and other modal frames satisfying particular axioms by analyzing their asymptotic properties.

\begin{Theorem}\label{gamenumber}
The number of distinct games on a set $T$ with cardinality $n$ is asymptotically
$$\sum_{i=1}^{n} \sum_{j=1}^{n-i} {n \choose i} {n-i \choose j} (2^i - 1)^{j}(2^{j}-1)^{n-i-j}$$
\end{Theorem}
\begin{proof}
Given a set of time points $T$, a strict partial ordering $<$ on $T$ defines a branching frame $\mathcal{F} = \langle T, < \rangle$. The relation $<$ determines the rules of the game, that is, $t_{i} < t_{j}$ if it possible for a player to make a move which alters the configuration of the game from $t_{i}$ to $t_{j}$. So, we have that the number of distinct games on a set $T$ with cardinality $n$ is equal to the number of strict partial orderings of $T$. Let $\leq$ be a non-strict partial order then define the reflexive reduction of $\leq$ by $a < b$ if and only if $a \leq b$ and $a \neq b$. Conversely, let $<$ be a strict partial order and define the reflexive closure of $<$ by $a \leq b$ if and only if $a < b$ or $a=b$. Therefore, there exists a bijection between the set of all strict partial orders and non-strict partial orders and so the number of strict partial orders on a set of $n$ elements is equal to the number of non-strict partial orders on a set of $n$ elements.

By \cite{asymp}, we have that the number of partial orders on a set of $n$ elements is equal to
\begin{align*}
&\left(1 + O\left(\frac{1}{n}\right)\right)\left(\sum_{i=1}^{n} \sum_{j=1}^{n-i} {n \choose i} {n-i \choose j} (2^i - 1)^{j}(2^{j}-1)^{n-i-j}\right)  \\
&= \left(\sum_{i=1}^{n} \sum_{j=1}^{n-i} {n \choose i} {n-i \choose j} (2^i - 1)^{j}(2^{j}-1)^{n-i-j}\right) + O\left(\frac{1}{n}\right)\left(\sum_{i=1}^{n} \sum_{j=1}^{n-i} {n \choose i} {n-i \choose j} (2^i - 1)^{j}(2^{j}-1)^{n-i-j}\right) \\
&\sim \sum_{i=1}^{n} \sum_{j=1}^{n-i} {n \choose i} {n-i \choose j} (2^i - 1)^{j}(2^{j}-1)^{n-i-j}
\end{align*}
Therefore, since the number of distinct games on $T$ is equal to the number of strict partial orderings of $T$ which is equal to the number of non-strict partial orders of $T$, the number of distinct games on a set $T$ with cardinality $n$ is asymptotically
$$\sum_{i=1}^{n} \sum_{j=1}^{n-i} {n \choose i} {n-i \choose j} (2^i - 1)^{j}(2^{j}-1)^{n-i-j}$$
\end{proof}

A generalization of this theorem can be provided for other modal axioms being satisfied in a Kripke frame $\mathcal{F}$. That is, we specifically restricted our attention to branching frames $\mathcal{F} = \langle T, < \rangle$ where $<$ is a strict partial order. Yet, we can provide asymptotic bounds on other characteristics of Kripke frames based on the number of relations on a set of $n$ elements of a certain type. For example, consider $S5$ frames, that is, frames $\mathcal{F} = \langle W,R \rangle$ where $W$ is the set of all possible world and $R$ is a relation, for which there exists a valuation that satisfies the following three axioms at every $w \in W$:
\begin{align*}
&\textbf{K}: \Box(\varphi \rightarrow \psi) \rightarrow (\Box \varphi \rightarrow \Box \psi) \\
&\textbf{T}: \Box \varphi \rightarrow \varphi \\
&\textbf{5}: \Diamond \varphi \rightarrow \Box \Diamond \varphi
\end{align*}
S$5$ modal frames can alternatively be characterized by $\textbf{K}$ and $\textbf{T}$ along with two additional axioms
\begin{align*}
&\textbf{4}: \Box \varphi \rightarrow \Box \Box \varphi \;\;\;\;\;\;\;\;\;\;\;\;\;\;\;\;\;\;\;\;\;\;\; \\
&\textbf{B}: \varphi \rightarrow \Box \Diamond \varphi
\end{align*}

This equivalent characterization of S$5$ modal frames follows from the following theorem. It is a standard result in modal logic that S$5$ can be alternatively characterized this way and a proof can be found in \cite{frameref}.

\begin{Theorem}\label{54Btheorem}
Assume $\mathcal{F}$ is a normal modal frame, that is, $\mathcal{F} \models \Box(\varphi \rightarrow \psi) \rightarrow (\Box \varphi \rightarrow \Box \psi)$. Then $\mathcal{F} \models \Diamond \varphi \rightarrow \Box \Diamond \varphi$ and $\mathcal{F} \models \Box \varphi \rightarrow \Box \Box \varphi$ if and only if $\mathcal{F} \models \Diamond \varphi \rightarrow \Box \Diamond \varphi$.
\end{Theorem}

Now, using the Scott-Lemmon result \cite{Lemmon1977}, we can show a correspondence between S$5$ modal frames and equivalence relations. This correspondence will be essential in our combinatorial analysis of S$5$ modal frames.

\begin{Lemma}\emph{(Scott-Lemmon)} \\
In any frame $\mathcal{F} = \langle W, R \rangle$ with $w_{1},...,w_{n-1} \in W$, 
$F \models \Diamond^{h} \Box^{i} p \rightarrow \Box^{j} \Diamond^{k} p$ implies that for arbitrary $u,v,w \in T$, $w R^{h} v \wedge w R^{j} u \rightarrow \exists x(v R^{i} x \wedge u R^{k} x)$, where the relation is defined by $$a R^{n} b = \bigexists_{i=1}^{n-1} x_{i}\left(a R x_{1} \wedge \left(\bigwedge_{i=1}^{n-2} x_{i} R x_{i+1} \right) \wedge x_{n-1} R b\right)$$
if $n \geq 2$, by $a R b$ if $n=1$, and by $a = b$ if $n = 0$.
\end{Lemma}

We now give a proof using the Scott-Lemmon result that S$5$ modal frames can be thought of as frames with equivalence relations as the accessibility relation between possible worlds.

\begin{Theorem}\label{s5equiv}
Every S$5$ modal frame is reflexive, symmetric, and transitive.
\end{Theorem}
\begin{proof}
Let $\mathcal{F} = \langle W,R \rangle$ be an arbitrary S$5$ modal frame. That is, assume $\mathcal{F} \models \Box(\varphi \rightarrow \psi) \rightarrow (\Box \varphi \rightarrow \Box \psi), \mathcal{F} \models \Box \varphi \rightarrow \varphi,$ and $\mathcal{F} \models \Diamond \varphi \rightarrow \Box \Diamond \varphi$. Firstly, we claim that $\mathcal{F}$ is reflexive. By the Scott-Lemmon result, $\mathcal{F} \models \Box \varphi \rightarrow \varphi = \Diamond^{0} \Box^{1} \varphi \rightarrow \Box^{0} \Diamond^{0} \varphi$ implies $w R^{0} \wedge w R^{0} u \rightarrow \exists x(v R x \wedge u R^{0} x)$ which is that $(w = v \wedge w =u ) \rightarrow \exists x (v R x \wedge u =x)$. This reduces to $v = u \rightarrow v R u$, hence $\mathcal{F}$ is reflexive since $u,v \in W$ are arbitrary. By Theorem \ref{54Btheorem}, we have $\mathcal{F} \models \Diamond \varphi \rightarrow \Box \Diamond \varphi$ and $\mathcal{F} \models \Box \varphi \rightarrow \Box \Box \varphi$. Secondly, we claim that $\mathcal{F}$ is symmetric. By the Scott-Lemmon result, $\mathcal{F} \models \varphi \rightarrow \Box \Diamond \varphi = \Diamond^{0} \Box^{0} \varphi \rightarrow \Box^{1} \Diamond^{1}$ implies $w R ^{0} v \wedge w R^{1} u \rightarrow \exists x(v R^{0} x \wedge u R^{1} x)$ which is that $w R u \rightarrow u R w$, hence $\mathcal{F}$ is symmetric since $u,w \in W$ are arbitrary. Lastly, we claim that $\mathcal{F}$ is transitive. By the Scott-Lemmon result we have that $\mathcal{F} \models \Box \varphi \rightarrow \Box \Box \varphi = \Diamond^{0} \Box^{1} \varphi \rightarrow \Box^{2} \Diamond^{0} \varphi$ implies $w R^{0} v \wedge w R^{2} u \rightarrow \exists x(v R x \wedge u R^{0} x)$ which is that $w = v \wedge \exists y (w R y \wedge y R u) \rightarrow \exists x (v R x \wedge u = x)$, so $\exists y(w R y \wedge y R u ) \rightarrow w R u$. Hence $\mathcal{F}$ is transitive since $u,w \in W$ are arbitrary. Therefore, every S$5$ modal frame is reflexive, symmetric, and transitive.
\end{proof}

We remark that the number of isomorphism classes of S$5$ modal frames is the number of non-isomorphic S$5$ modal frames, by the definition of isomorphism class. The following lemma is necessary in order to study the combinatorics of non-isomorphic S$5$ modal frames.

\begin{Lemma}\label{isolemma}
Let $\mathcal{F} = \langle W,R \rangle$ and $\mathcal{F}' = \langle W',R' \rangle$ be S$5$ modal frames. Then, $\mathcal{F} \cong \mathcal{F}'$ if and only if there exists a bijection $\varphi: W \rightarrow W'$ such that $u R v \Longleftrightarrow \varphi(u) R' \varphi(v)$, $\forall u,v \in W$.
\end{Lemma}
\begin{proof}
An isomorphism $\mathcal{F} \cong \mathcal{F}'$ induces, and is induced by, a bijection $\varphi: W \rightarrow W'$ for which we can construct a bijection $\varphi \times \varphi: W \times W \rightarrow W' \times W'$ defined by $(\varphi \times \varphi)(u,v) = (\varphi(u),\varphi(v))$.
\end{proof}

If we have that there exists a bijection $\phi: W \rightarrow W'$ such that $u R v \Longleftrightarrow \phi(u) R' \phi(v)$, $\forall u,v \in W$, then we say that the two relations are isomorphic and write $R \cong R'$. We now have the sufficient background to provide our combinatorial analysis of S$5$ modal frames as a generalization of Theorem \ref{gamenumber}.

\begin{Theorem}\label{framenumber}
The number of non-isomorphic $S5$ modal frames $\mathcal{F} = \langle W, R \rangle$ with $|W|=n$, denoted by $|\mathcal{F}(n)|_{S5}$, is
$$\frac{1}{\pi\sqrt{2}} \sum_{k=1}^{\infty} \sum_{h=1}^{k} \delta_{\gcd(h,k),1} \text{exp}\left(\pi i \sum_{j=1}^{k-1} \frac{j}{k}\left(\frac{hj}{k} - \left\lfloor \frac{hj}{k} \right\rfloor - \frac{1}{2}\right) - \frac{2\pi i h n}{k} \right) \sqrt{k} \frac{d}{dn}\left[ \frac{\sinh\left(\frac{\pi}{k} \sqrt{\frac{2}{3}(n - \frac{1}{24})}\right)}{\sqrt{n - \frac{1}{24}}} \right]$$
\end{Theorem}
\begin{proof}
By Theorem \ref{s5equiv} and Lemma \ref{isolemma} we have that there exists a bijection between the set of all S$5$ modal frames $\mathcal{F} = \langle W,R \rangle$ with $|W|=n$ and the set of all equivalence classes on $n$ elements. The number of non-isomorphic equivalence relations on a set of $n$ elements is the number of integer partitions of $n$ from page 57 in \cite{modelsandgames}, so $|\mathcal{F}(n)|_{S5}$ is equal to $p(n)$, the number of integer partitions of $n$. Therefore, the theorem holds due to the expression for $p(n)$ given by Rademacher \cite{1937}.

\end{proof}

Clearly the equation in Theorem \ref{framenumber} is unwieldy for any computation for a particular value of $n$, so we mention the following corollary.

\begin{Corollary}\label{asympframe}
The number of non-isomorphic S$5$ modal frames $\mathcal{F} = \langle W,R \rangle$ with $|W| = n$ is asymptotically the Hardy-Ramanujan number
$$\frac{1}{4\sqrt{3}n}e^{\pi \sqrt{2n/3}}$$
\end{Corollary}
\begin{proof}
Due to Erd\"{o}s \cite{erdos} we have that $$p(n) \sim \frac{1}{4\sqrt{3}n}e^{\pi \sqrt{2n/3}}$$
and therefore by Theorem \ref{framenumber} we have
$$|\mathcal{F}(n)|_{S5} \sim \frac{1}{4\sqrt{3}n}e^{\pi \sqrt{2n/3}}$$
\end{proof}

We now have the results necessary to prove that a modal frame is almost surely not an S$5$ modal frame.

\begin{Theorem}\label{zeroprob}
An arbitrary modal frame $\mathcal{F}$ is an S$5$ modal frame with probability zero.
\end{Theorem}
\begin{proof}
Let $\mathcal{F} = \langle W,R \rangle$ be an arbitrary modal frame with $|W|=n$. By \cite{integerseq}, we have that the number of non-isomorphic relations on a set of $n$ elements is
$$a(n) = \sum_{1s_{1} + 2s_{2} + \cdot\cdot\cdot =n} \left(\frac{2^{\left(\displaystyle\sum_{i=1} \displaystyle\sum_{j=1} \gcd(i,j)s_{i}s_{j}\right)}}{\displaystyle\prod_{k=1} k^{s_{k}}s_{k}!}\right)$$
Since the number of relations on $W$ is $2^{n^2}$ and there are $n!$ bijections from $W$ to $W$, the number of non-isomorphic relations on $W$ is asymptotically $$\frac{2^{n^2}}{n!} \sim 2^{n^2 - \varepsilon}$$ for some $\epsilon >0$. So, since we have that $2^{n^2 - \epsilon} < a(n)$, Corollary \ref{asympframe} implies that
$$\lim_{n \rightarrow \infty} \frac{p(n)}{a(n)} \leq \lim_{n \rightarrow \infty} \frac{p(n)}{\frac{2^{n^2}}{n!}} = \lim_{n \rightarrow \infty} \frac{\frac{1}{4\sqrt{3}n}e^{\pi \sqrt{2n/3}}}{2^{n^2 - \varepsilon}} = 0$$
Hence,
$$\lim_{n \rightarrow \infty} \frac{p(n)}{a(n)} = 0$$
Thus, since Theorem \ref{framenumber} says that $p(n)$ is the number of non-isomorphic S$5$ modal frames with $n$ possible worlds, and $a(n)$ is the number of non-isomorphic relations on a set of $n$ elements, we have that the probability that $\mathcal{F}$ is an S$5$ modal frame is zero.
\end{proof}

We remark that despite the independent interest of the statement and proof of Theorem \ref{zeroprob} as an application of analytic number theory to modal logic, potential philosophical applications of Theorem \ref{zeroprob} exist. One possible example could be a critique of Alvin Plantinga's version of the ontological argument \cite{plantinga} by noticing that the penultimate conclusion, ``it is necessarily true that an omniscient, omnipotent and perfectly good being exists.'', assumes S$5$ modal logic.


\begin{thebibliography}{10}
\bibitem{modelsandgames}
Jouko V\"{a}\"{a}n\"{a}nen.
\newblock {\em Models and Games}.
\newblock Cambridge University Press, United Kingdom, 2011.

\bibitem{asymp}
D. J. Kleitman and B. L. Rothschild.
\newblock Asymptotic Enumeration of Partial Orders on a Finite Set.
\newblock {\em Transactions of the American Mathematical Society}, 205(5): 205--220, 1975.

\bibitem{Schwalbe2001123}
Ulrich Schwalbe and Paul Walker.
\newblock Zermelo and the Early History of Game Theory.
\newblock {\em Games and Economic Behavior}, 34(1): 123--137, 2001.

\bibitem{Lemmon1977}
E. J. Lemmon.
\newblock {\em An Introduction to Modal Logic: The Lemmon Notes}.
\newblock B. Blackwell, 1977.

\bibitem{erdos}
Erd\H{o}s, P\'aul.
\newblock On an elementary proof of some asymptotic formulas in the theory of partitions.
\newblock {\em Annals of Mathematics}, 43: 437--450, 1942.

\bibitem{1937}
Hans Rademacher.
\newblock A Convergent Series for the Partition Function p(n).
\newblock {\em Proceedings of the National Academy of Sciences of the United States of America}, 23(2): 78--84, 1937.

\bibitem{Lovasz}
L. Lovász.
\newblock {\em Combinatorial Problems and Exercises, 2nd ed.}.
\newblock North-Holland, Amsterdam, 1993.

\bibitem{integerseq}
N. J. A. Sloane and Simon Plouffe.
\newblock {\em The Encyclopedia of Integer Sequences}.
\newblock Academic Press, 1995.

\bibitem{frameref}
G.E. Hughes and Max J. Cresswell.
\newblock {\em A New Introduction to Modal Logic}.
\newblock Routledge, 1996.

\bibitem{plantinga}
A. Plantinga.
\newblock {\em God and Other Minds}.
\newblock Cornell University Press, 1967.
\end{thebibliography}
\end{document}